\documentclass[a4paper,12pt]{article}
\usepackage[utf8]{inputenc}
\usepackage{amsmath,amssymb,amsthm}
\usepackage{tikz-cd}
\usepackage{xcolor}
\usepackage{mathtools}
\usepackage{esint}
\usepackage{setspace}
\usepackage{hyperref}
\newcommand{\mz}{\mathbb{Z}}
\newcommand{\mr}{\mathbb{R}}
\newcommand{\mrz}{\mathbb{R}/\mathbb{Z}}

\newcommand{\dr}{\mathbb{D}^r}
\newcommand{\Glnr}{GL(n,\mathbb{R})}
\newcommand{\glnr}{\mathfrak{gl}(n,\mathbb{R})}
\newcommand{\hhat}{\hat{H}}

\newcommand{\fintsigma}{\underset{\Sigma}{\fint}}
\newcommand{\fintf}{\underset{F}{\fint}}
\newcommand{\dcfib}[1]{\hat{\pi}_{!#1}}
\newcommand{\DE}{\mathcal{D}(E)}
\newcommand{\SigmaiB}{\Sigma^i \times B}

\newcommand{\mcF}{\mathcal{F}}
\newcommand{\mcA}{\mathcal{A}}

\newtheorem{theorem}{Theorem}
\newtheorem{corollary}[theorem]{Corollary}
\newtheorem{lemma}[theorem]{Lemma}
\newtheorem{definition}[theorem]{Definition}
\newtheorem{remark}[theorem]{Remark}

\title{Invariants of families of flat connections using fiber integration of differential characters}
\author{Ishan Mata}
\date{}

\begin{document}
\maketitle

\begin{abstract}
Let $E\to B$ be a smooth  vector bundle of rank $n$, and let $P \in I^p(GL(n,\mathbb{R}))$ be a $GL(n,\mathbb{R})$-invariant polynomial of degree $p$ compatible with a universal integral characteristic class $ u \in H^{2p}(BGL(n,\mathbb{R}),\mathbb{Z})$. Cheeger-Simons theory associates a rigid invariant in $H^{2p-1}(B,\mathbb{R}/\mathbb{Z})$ to any flat connection on this bundle. Generalizing this result, Jaya Iyer (\textit{Letters in Mathematical Physics}, 2016, 106 (1) pp. 131-146) constructed maps  $H_r(\mathcal{D}(E)) \to H^{2p-r-1}(B,\mathbb{R}/\mathbb{Z})$ for $p>r+1$ where $\mathcal{D}(E)$ is the simplicial set of relatively flat connections, thereby associating invariants to families of flat connections. In this article we construct such maps for the cases $p<r$ and $p>r+1$ using fiber integration of differential characters. We find that for $p>r+1$ case, the invariants constructed here coincide with those obtained by Jaya Iyer, and that in the $p<r$ case the invariants are trivial. We further compare our construction with other results in the literature.
\end{abstract}

\section{Introduction}
Fix a smooth vector bundle $E \to B$ of rank $n$ on a finite dimensional smooth base manifold $B$. Let $\omega$ be a smooth connection on this bundle, and denote its curvature by $\Omega \in \Omega^2(B,End(E))$. Denote the set of degree $p$ $\Glnr$-invariant polynomials on $\glnr$ by $I^p(\Glnr)$. Chern-Weil theory (see \cite{tu}, for example) gives a homomorphism $I^p(\Glnr) \to H^{2p}(B,\mr)$, given by $P \mapsto [P(\Omega)]$ for any choice of connection on the bundle where $[\ ]$ denotes the equivalence class in the cohomology. By $I^p_\mz(\Glnr)$, we denote that subset of $I^p(\Glnr)$ whose image under the Chern-Weil map for the classifying bundle $E\Glnr \to B\Glnr$ lies in $H^{2p}(B\Glnr, \mz)$. Here by $H^{2p}(B\Glnr, \mz)$ we really mean its image in $H^{2p}(B\Glnr, \mr)$. Let $u \in H^{2p}(BGL(n,\mr),\mz)$ be a universal integral characteristic class compatible with $P$.
Cheeger and Simons \cite{cs2} (also see \cite{cs1}) defined objects \begin{equation}
                                               \hhat^k(B) := \{f: Z_{k-1}(B) \to \mr/\mz | f \circ \delta \in \Omega^k(B)\}.
                                              \end{equation}
Notice that our convention of the degree is different from theirs : they use $\hhat^{k-1}(B)$ to denote the R.H.S.. They showed that these objects fit into exact sequences 
\begin{equation}\label{csexact1}
 0 \to H^{k-1}(B,\mrz) \xrightarrow{i_1} \hhat^{k}(B) \xrightarrow{curv} \Omega^k_{cl}(B) \to 0
\end{equation}
\begin{equation}\label{csexact2}
 0 \to \frac{\Omega^{k-1}(B)}{\Omega^{k-1}_0(B)} \xrightarrow{\iota} \hhat^{k}(B) \xrightarrow{ch} H^{k}(B,\mz) \to 0.
\end{equation}
Here $\Omega_0^{k-1}(B)$ denotes closed $(k-1)$-differential forms with integral periods. They showed that to the data $(E \to B, \omega)$, one can associate a differential character $h \in \hhat^{2p}(B)$, such that $curv(h)=P(\Omega)$, and $ch(h)$ is the $u$-characteristic class of the bundle $E \to B$. As a corollary, they obtain the result that when the connection is flat the differential character lies in the image of the inclusion of $H^{2p-1}(B,\mrz)$ in $\hhat^{2p}(B)$. Thus the authors determined $\mrz$ cohomology classes which we denote by $cs_{(P,u)}(E,\theta) \in H^{2p-1}(B,\mrz) $. 
Generalizing this  work to families of connections, Jaya Iyer \cite{jni} showed how to associate an element of $H^{2p-r-1}(B,\mrz)$ to an element of $r$-th simplicial homology of the simplicial abelian group of relatively flat connections on the bundle. More precisely, she defines the simplicial abelian group $\DE$ whose set of $r-$simplices is the free abelian group generated by $(r+1)$-tuples $(D^0,\cdots,D^r)$ of relatively flat connections. Relative flatness means that $\underset{j}{\Sigma} t_j D^i$ is flat for any choice of $t_i$'s such that $\Sigma t_j =1$. She then constructs maps $\rho_{p,r}: H_r(\DE) \to H^{2p-r-1}(B,\mrz)$ for $p > r+1, r\geq 1$.\\
This article is a humble and modest extension of Jaya Iyer's work (henceforth the article \cite{jni} is often referred to as `Jaya Iyer's paper') using the technique of fiber integration of differential characters developed by B\"ar and Becker in \cite{bb}. Given an element $u$ compatible with $P$ as above, we derive maps $\tilde{\psi}_{P,u,r} : H_r(\DE) \to H^{2p-r-1}(B,\mrz)$  for $p \neq r,r+1$. If $\Sigma$ is an $r$-cycle in $\DE$, we consider the bundle $E \times \Sigma \to B \times \Sigma$ (we use the same symbol $\Sigma$ to denote the cycle and its geometric realization), and endow it with a certain connection. We apply the Cheeger-Simons theory to this data to obtain a differential character $h_{B \times \Sigma} \in \hhat^{2p}(B \times \Sigma)$, and apply integration along the fibers of the bundle $B \times \Sigma \to B$ to get a differential character in $\hhat^{2p-r}(B)$. We find that these maps do not actually depend on the choice of $u \in H^{2p}(BGL(n,\mr),\mz)$, and that the characteristic class of $\tilde{\psi}_{P,u,r}([\Sigma])$ considered as a differential character vanishes. Thus we conclude that $\tilde{\psi}_{P,u,r}([\Sigma])$  is in the image of the map $\frac{\Omega^{2p-r-1}}{\Omega^{2p-r-1}_0} \xrightarrow{\iota} \hhat^{2p-r}(B)$. Using again a theorem of B\"ar and Becker \cite{bb}, we compute a representative differential form. We find that for the $p > r+1$ case, this form matches with the one constructed in \cite{jni}. Our initial hope was to obtain new cohomology invariants in the $p<r$ case. However, it turns out that in this case the invariants are trivial i.e. $\tilde{\psi}(\Sigma)=0$.\\
Thereafter we proceed to discuss the relationship of our construction to other constructions in the literature. In \cite{bl} Biswas and Lopez consider the set of smooth maps $Maps(S,\mathcal{F})$ where $S$ is a smooth null-cobordant manifold of dimension $r$, and $\mathcal{F}$ is the space of flat connections on the principal G-bundle $E \to B$. Using the formalism of Atiyah bundle and the bundle of connections, they define certain forms $\beta^p_k \in \Omega^k(\mathcal{A},\Omega^{2r-k}(B))$. Here $\mathcal{A}$ denotes the space of all connections on $E \to B$ considered as an infinite dimensional Fr\'echet manifold. Using these forms they construct maps $\Lambda^p_{r+1} : Maps (S,\mathcal{F}) \to H^{2p-r-1}(B,\mr)$. They prove that the maps can be described as $[f: S \to \mathcal{F}] \mapsto [\underset{T}{\fint}P(\Omega)]$ where $\bar{f} : T \to \mathcal{A}$ is an extension of $f$ to a manifold $T$ whose boundary is $S$, and $\Omega$ is the curvature of a certain connection on $E \times T \to B \times T$. They show that these maps are well defined i.e. the answer does not depend on the choice of the extension $\bar{f}$ or the manifold $T$, that $\Lambda^p_{r+1}$ are indeed closed forms which define elements of the cohomology group $H^{2p-r-1}(B,\mr)$, and that $[\Lambda^p_{r+1}(f_0)]=[\Lambda^p_{r+1}(f_1)] \in H^{2p-r-1}(B,\mr)$ whenever $f_0,f_1$ are homologous. As argued in section \ref{comparison}, from the viewpoint of fiber integration these results proved in \cite{bl} can be obtained (modulo $\mz$) directly as consequences of some properties of fiber integration proved in \cite{bb}.\\
In \cite{lp}, the authors consider the principal G-bundle $E \times \mathcal{A} \to B \times \mathcal{A}$ with its canonical connection $\mathbb{A}$. If $\mathcal{G}$ is a subgroup of $Gau(E)$ which acts freely on $\mathcal{A}$, then $\mathcal{A} \to \mathcal{A}/\mathcal{G}$ becomes a principal bundle. They show that a choice of connection $\mathcal{U}$ on $\mathcal{A} \to \mathcal{A}/\mathcal{G}$ gives a connection $\underline{\mathcal{U}}$ on $(E \times \mathcal{A})/\mathcal{G} \to B \times \mathcal{A}/\mathcal{G}$, and  apply Cheeger-Simons theory to this bundle to obtain maps $\chi_{r} : H_r({\mathcal{F}/\mathcal{G}},\mz) \times H_{2p-r-1}(B,\mz) \to \mrz$ which do not depend on the choice of the connection $\mathcal{U}$. They also show that their approach using differential characters yields the same results as in \cite{bl} on cycles in $\mathcal{F}/\mathcal{G}$ that come from cycles in $\mathcal{F}$.\\
Our results can be considered a special case of their results for the case $\mathcal{G}=\{e\}$. (Even though we state and prove our results for vector bundles, the same discussion applies \textit{mutatis mutandis} to principal $G-$bundles.) However, since we deal with finite dimensional stratifolds, we do not need to assume that the Cheeger-Simons construction of a differential character given a smooth connection on a bundle holds good in case the base of the bundle is an infinite dimensional Fr\'echet manifold (see section \ref{comparison} for a discussion of this point). Also our method makes the relation between various approaches used in \cite{bl,lp,jni}, and maps obtained therein  more transparent. We find that all the three approaches yield the same invariant in $\mrz$ cohomology upto a possible sign factor (which depends on the orientation conventions).\\
The idea that fiber integration of differential characters can be used to obtain invariants of families of connections was explored in \cite{freed2} (also see \cite{ljungmann}). However as observed in \cite{jni}, the fiber integration developed in \cite{freed2} can not be used for the map $X \times \Delta^r \to X$ since it assumes that the fibers are compact manifolds without boundary. In our case, as we shall see, the fibers are not manifolds but compact stratifolds without boundary. This allows us to make use of the fiber integration construction of \cite{bb}. In fact our results are a direct application of the fiber integration construction (lemma 41), and proposition 54 of \cite{bb}.\\ 
The article is organized as follows. In section \ref{prelim} we very briefly describe the geometric chain model of differential characters, and fiber integration developed by B\"ar and Becker in \cite{bb}. In section \ref{construction} we give our construction of the invariants, and  thereafter in section \ref{comparison} we proceed to compare our construction with those in \cite{bl,lp,jni}. \\
The subject of invariants of flat connections is a subject of many articles, see for example \cite{guruprasad,ljungmann,dupont}.
\section{Preliminaries : Differential characters, geometric chains, and fiber integration}\label{prelim}
In this section we briefly state the definitions and results that are used later in the paper. Subsection \ref{geometricchains} describes a geometric chain model of differential characters on smooth spaces, and subsection \ref{fiberintegration} discusses the construction of fiber integration of differential characters and its properties.\\
Familiarity with smooth spaces and stratifolds is assumed, see section 2 of \cite{bb} for statement of essential definitions and results, and \cite{kreck} for an authoritative and comprehensive treatment. The discussion in this section is based on the framework developed in \cite{bb}. Nothing contained in this section is original.
\subsection{Geometric chain model of differential characters}\label{geometricchains}
\begin{definition}
 Let $\mathcal{C}_k(X)$ be the set of equivalence classes of smooth maps $\zeta : M \to X$ where $M$ is an oriented $k-$stratifold  such that $\partial M$ is an oriented $(k-1)-$stratifold under the following equivalence relation : \\ 
 $(\zeta : M \to X) \sim (\zeta': M' \to X) $ iff $\exists$ an orientation preserving diffeomorphism $\psi : M \to M'$ such that $\zeta' \circ \psi = \zeta$. The operation of disjoint union makes $\mathcal{C}_k(X)$ an abelian semi-group. We define the inverse of $[\zeta : M \to X]$ to be the same map but with the orientation of $M$ reversed : $[\bar{\zeta} : \bar{M} \to X]$. In this manner, $\mathcal{C}_k(X)$ becomes an abelian group. Elements of $\mathcal{C}_k(X)$ are called geometric chains.
\end{definition}
\begin{definition}
 The boundary operator $\partial : \mathcal{C}_k(X) \to \mathcal{C}_{k-1}(X)$is defined as $\partial [\zeta : M \to X] = [\zeta|_{\partial M} : \partial M \to X]$.
\end{definition}
\begin{definition}
 The set ker$\partial := \mathcal{Z}_k(X)$ is called the group of $k-$geometric cycles. The set im$\partial := \mathcal{B}_{k-1}(X)$ is called the group of geometric boundaries. The quotient $\mathcal{H}_k(X) := \frac{\mathcal{Z}_k(X)}{\mathcal{B}_k(X)}$ is the homology of the complex defined by $\partial$.
\end{definition}

Let $C_n(X;\mz)$ denote the group of smooth singular $n-$chains on $X$. A chain $c \in C_n(X;\mz)$ is called thin if $\forall \omega \in \Omega^n(X), \int_c \omega =0$. We denote the group of thin $n-$chains by $S_n(X, \mz)$. \\
We now define maps $\psi : \mathcal{C}_n(X) \to C_n(X,\mz)/S_n(X,\mz)$ by $ [\zeta : M \to X] \mapsto [\zeta_*(c)]_{S_n}$ where $c$ denotes a fundamental cycle of $H_n(M,\partial M, \mz)$ (or $H_n(M,\mz)$ if $\partial M = \phi$). Similarly we have maps from geometric cycles, and geometric boundaries to smooth singular cycles, and smooth singular boundaries respectively. We have the commutative diagram : 

\begin{tikzcd}
{} \arrow[r] & {\mathcal{C}_{n+1}(X)} \arrow[rr] \arrow[dd] &  & {\mathcal{B}_n(X)} \arrow[rr] \arrow[dd] &  & {\mathcal{Z}_n(X)} \arrow[rr] \arrow[dd] &  & {\mathcal{C}_n(X)} \arrow[dd] \arrow[r] & {} &  \\
 &  &  &  &  &  &  &  &  &  \\
{} \arrow[r] & {\frac{C_{n+1}(X,\mathbb{Z})}{S_{n+1}(X,\mathbb{Z})}} \arrow[rr] &  & {\frac{B_{n}(X,\mathbb{Z})}{\partial S_{n+1}(X,\mathbb{Z})}} \arrow[rr] &  & {\frac{Z_n(X,\mathbb{Z})}{\partial S_{n+1}(X,\mathbb{Z})}} \arrow[rr] &  & {\frac{C_n(X,\mathbb{Z})}{S_n(X,\mathbb{Z})}} \arrow[r] & {} & {}
\end{tikzcd}

This chain map induces a map on homology of the two chain complexes : $\mathcal{H}_n(X) := \frac{\mathcal{Z}_n(X)}{\mathcal{B}_n(X)} \to H_n(X,\mz)$. This map is an isomorphism (see Theorem 20.1 in \cite{kreck}). Further it is shown there that the product $\mathcal{H}_m(X) \times \mathcal{H}_n(Y) \to \mathcal{H}_{m+n}(X \times Y)$ given by $[\zeta : M \to X]\times[\eta : N \to Y] \mapsto [\zeta \times \eta : M \times N \to X \times Y]$ is compatible with the isomorphism above and the usual multiplication in smooth singular homology.

B\"ar and Becker show (lemma 7 of \cite{bb}) that 
\begin{theorem}\label{zetamaps}
There exist homomorphisms $\zeta : C_{n+1}(X,\mz) \to \mathcal{C}_{n+1}(X)$, $a : C_n(X,\mz) \to C_{n+1}(X,\mz)$, and $y : C_{n+1}(X,\mz) \to Z_{n+1}(X, \mz)$ such that the following hold :
\begin{equation}
 \partial \zeta (c) = \zeta \partial (c) \ \ \ \forall c \in C_{n+1}(X,\mz),
\end{equation}
\begin{equation}
 [\zeta(c)]_{S_{n+1}} = [c-a(\partial c) -\partial a(c+y(c))]_{S_{n+1}} \ \ \ \ \forall c \in C_{n+1}(X,\mz),
\end{equation} and 
\begin{equation}
 [\zeta(z)]_{\partial S_{n+1}} = [z-a(z)]_{\partial S_{n+1}} \ \ \ \ \forall z \in Z_{n+1}(X,\mz).
\end{equation}
\end{theorem}

\subsection{Fiber integration of differential characters}\label{fiberintegration} 
Before discussing fiber integration, let us fix our conventions regarding orientation. We use the same conventions as in \cite{bb}. In particular, if $M$ is a manifold with boundary $\partial M$, a tangent vector pointing outward at a boundary point $p \in \partial M$ followed by an oriented basis for $T_p(\partial M)$ gives us an oriented basis for the manifold $M$ at $p$. For a fiber bundle with oriented fiber and oriented base, the orientation on the total space is chosen to be given by an oriented basis of the base followed by an oriented basis of the fiber.
We now discuss fiber bundles over smooth spaces. As remarked in \cite{bb} (sec 7.1), there are multiple non-equivalent generalizations of the concept of fiber bundles over smooth spaces. We use the same definition as by the above authors. 
\begin{definition}
 A smooth surjective map $p : E \to B$ is a fiber bundle with fiber $F$, if for any smooth map $f : M \to B $ from a finite dimensional stratifold $M$, the pull back $f^*E \to M $ is locally trivial with fiber $F$.
\end{definition}
Let $F \xhookrightarrow{} E \to B $ be a fiber bundle where $M,F,$ and $E$ are smooth manifolds of finite dimension and $F$ is compact oriented. Then we have a fiber integration map of differential forms (see, for example, \cite{botttu}) \\ $\underset{F}{\fint} : \Omega^n(E,\mr) \to \Omega^{n-r}(B)$ for $n \geq r$ where $dim F =r$. This map satisfies the Stokes theorem :
\begin{equation}\label{stokesthm}
\fintf d \omega = d \fintf \omega + (-1)^{deg \omega + dim F} \underset{\partial F}{\fint} \omega                                                                                                                                                                                                                                                    
                                                                                                                                                                                                                                                                                                            \end{equation}
                                                                                                                                                                                                                                                                                                            Similarly there is a push-forward map (see \cite{borelhirzebruch,chern}) for singular cohomology with coefficients in an arbitrary group $\pi_F : H^{n}(E,G) \to H^{n-r}(B,H^r(F,G))$. Generally $F$ is a connected, closed and oriented manifold of dimension $r$, and the above map becomes $\pi_F : H^{n}(E,G) \to H^{n-r}(B,G)$.
Fiber integration maps have been studied for various models of differential cohomology, see for example \cite{dupontljungmann,gomiterashima,ljungmann}. For our purpose, the construction given by \cite{bb} is suitable. We briefly describe this construction below. For this purpose, they use the transfer maps at the level of chains $\lambda : C_{k-r}(B,\mz) \to C_{k}(E,\mz)$ satisfying $[\lambda(z)]_{\partial S_{k+1}}=[PB_{E}(\zeta(z))]_{\partial S_{k+1}} \ \ \forall z \in Z_{k-r}(B,\mz)$ where $PB_E$ denotes the pull-back along the map $E \to B$.

\begin{definition}
 Let $E \to B$ be a fiber bundle with oriented closed fibers $F$, and let $dim F=r$. We then have fiber integration $\dcfib{F} : \hhat^k(E) \to \hhat^{k-r}(B)$ for $k > r$ given by $h \mapsto \dcfib{F}(h)$, where 
 \begin{equation}
  (\dcfib{F}(h))(z) = h(\lambda(z)) \times exp(2\pi i \underset{a(z)}{\int} \fintf curv(h))
 \end{equation}
\end{definition}
The authors of \cite{bb} show that this construction does not depend on the choice of functions $\lambda$,$\zeta$, and $a$.
\begin{remark}\label{sfold}
 The above construction in \cite{bb} is done for the case where $M$, and $E$ are smooth spaces, but $F$ is assumed to be a finite dimensional closed manifold. However, their construction as well as the properties cited below hold for the case when $F$ is a compact boundary-less finite dimensional stratifold. This is because all that is required in their construction and proofs is that (a) the fiber integration of differential forms is defined, and that (b) the integration of forms satisfies the Stokes theorem. These hold good when $F$ is a compact oriented stratifold by the virtue of results proved in \cite{coewald}.  
\end{remark}

When the fiber has a boundary, the fiber integration (along the boundary) of the restriction of a differential character, finds an expression in terms of the integral of the curvature of the differential character. This is a very useful identity, and in fact several of our results are a direct consequence of this formula. The precise theorem (proposition 54 of \cite{bb}) is as follows : 
\begin{theorem}\label{dcfibstokes}
 Let $F \xhookrightarrow \ E \to B$ be a fiber bundle where $F$ is a compact manifold with boundary $\partial F$, such that $\partial E \to B$ is a fiber bundle with fiber $\partial F$. If $h \in \hhat^{k}(E,\mz)$, then 
 \begin{equation}
  \dcfib{\partial F}(h|_{\partial E}) = \iota ((-1)^{k-dim F} \fintf curv(h))
 \end{equation}
\end{theorem}
As observed in \cite{bb} example 56, this is a generalization of the famous Cheeger-Simons homotopy formula (\ref{dchomotopyformula}).\\
They further show that fiber integration is compatible with the exact sequences i.e. the diagrams :\\ 
\begin{tikzcd}
0 \arrow[r] & \Omega^{k-1}(E)/\Omega^{k-1}_0(E) \arrow[rr] \arrow[dd, "\underset{F}{\fint}"] &  & \hhat^k(E) \arrow[rr] \arrow[dd, "\dcfib{F}"] &  & {H^k(E,\mz)} \arrow[r] \arrow[dd, "\pi_{!F}"] & 0 \\
 &  &  &  &  &  &  \\
0 \arrow[r] & \Omega^{k-r-1}(B)/\Omega^{k-r-1}_0(B) \arrow[rr] &  & \hhat^{k-r}(B) \arrow[rr] &  & {H^{k-r}(B,\mz)} \arrow[r] & 0
\end{tikzcd} 
\\and \\
\begin{tikzcd}
0 \arrow[r] & {H^{k-1}(E,\mr/\mz))} \arrow[rr] \arrow[dd, "\pi_{!F}"] &  & \hhat^k(E) \arrow[rr] \arrow[dd, "\dcfib{F}"] &  & \Omega^k_0(E) \arrow[r] \arrow[dd, "\underset{F}{\fint}"] & 0 \\
 &  &  &  &  &  &  \\
0 \arrow[r] & {H^{k-r-1}(B,\mr/\mz)} \arrow[rr] &  & \hhat^{k-r}(B) \arrow[rr] &  & \Omega^{k-r}_0(B) \arrow[r] & 0
\end{tikzcd}
\\ commute.

\section{Cohomological invariants of the space of flat connections using fiber integration of differential characters}\label{construction}

In their seminal paper \cite{cs2} Cheeger and Simons showed how to associate a differential character given a bundle equipped with a connection. We denote their differential character by $cs_{P,u}(E \to B, \theta) \in \hhat^{2p}(B)$. Here $P \in I_\mz^p(\Glnr)$, and $u\in H^{2p}(B\Glnr,\mz)$ are assumed to be compatible with each other. Often when there is no possibility of confusion, we simply write these as $cs(E,\theta)$. When the connection is flat, these characters are in the image of the inclusion $H^{2p-1}(B,\mrz) \to \hhat^{2p}(B)$. They further prove that for $p \geq2$, $cs(E,\theta_1)=cs(E,\theta_0)$ if the connections $\theta_0$ and $\theta_1$ are connected by a smooth family of flat connections and are therefore rigid invariants of the space of flat connections. This is a consequence of their 'homotopy formula' : 
\begin{equation}\label{dchomotopyformula}
 cs_{(P,u)}(\nabla_1) - cs_{(P,u)}(\nabla_0) = p \underset{I}{\int} P(\frac{d}{dt}\nabla_t,\nabla_t^2,\nabla_t^2,\cdots,\nabla_t^2).
\end{equation}
In this article, we consider the problem of attaching invariants to a family of flat connections. More precisely, we formulate the problem as done in \cite{jni} : Let $\DE$ be the simplicial abelian group, whose $r$-simplices are freely generated by $(r+1)-$tuples $(D^0,D^1,\cdots,D^r)$ of relatively flat connections on the bundle $E \to B$. (The connections $D^0,D^1,\cdots,D^r$ are called relatively flat if the linear combination $t_0D^0+t_1D^1+\cdots+t_rD^r$ is a flat connection for each choice of $t_0,\cdots,t_r$ such that $\sum t_i =1$.) If $\phi : [r] \to [s]$ is an increasing function (where $[r]:= \{ 0,1,\cdots,r\})$, we then define the corresponding map $\phi^* : \DE_s \to \DE_r$ by $\phi^* (D^0,D^1,\cdots,D^s) = (D^{\phi(0)},D^{\phi(1)}, \cdots, D^{\phi(r)} )$.
In this section we construct maps : \begin{equation}
                                     \psi_{p,r} : \mathbb{H}_r(\DE) \to H^{2p-r-1}(B,\mrz)  \ \ for \ p \neq r,r+1
                                    \end{equation}
using fiber integration of differential characters. In the next section we show that $p>r+1$ case, they agree with the maps constructed in \cite{jni}. Our original motivation for doing this construction was to obtain new invariants for the $p<r$ case, however we show in the next section that in this case, the invariants turn out to be trivial.\\
Let $\Sigma \in \mathcal{Z}_r(\DE)$ be a cycle representing the homology class $[\Sigma]\in \mathbb{H}_r(\DE)$. Then $\Sigma= \Sigma^1+\cdots+\Sigma^m$ where $\Sigma^i=(D^{0,i},D^{1,i},\cdots,D^{r,i})$. On the bundle $\Sigma^i \times E \to \Sigma^i \times B$, we define a connection as follows (we use the same symbol $\Sigma^i$ to denote the geometric realization of the chain $\Sigma^i$). $\Delta^r$ is conveniently parameterized by tuples $(t_0,\cdots,t_r)$ such that $t_0+t_1+\cdots+t_r=1$. Let $D^i=\underset{j}{\Sigma} t_j D^{j,i}$ be a connection on the bundle $E \times \Sigma^i \to B \times \Sigma^i$. \\
Schematically speaking, we could patch these connections $D^i$ on $E \times \Sigma^i \to B \times \Sigma^i$, to get a connection $D$ on $E \times \Sigma \to B \times \Sigma$ and then apply Cheeger-Simons theory to obtain a character $h_{B \times \Sigma} \in \hhat^{2p}(B \times \Sigma)$. We could then integrate along fibers of the bundle $B \times \Sigma \to B$ to get a character in $\hhat^{2p-r}(B)$. Broadly, this is the idea used in this paper. However carrying it out rigorously requires some care since $\Sigma$ is a stratifold, and not a manifold in general. In order to apply Cheeger-Simons theory to the bundle $E \times \Sigma \to B \times \Sigma$, we first need to ensure that the connection we endow it with is smooth, and that it is the pull back (under some smooth classifying map) of the universal connection on an $N$-classifying bundle (in the sense of Narsimhan-Ramanan \cite{narsimhanramanan}). To do this, we need to modify the connection $D^i$ so that it is constant in a small collar near the boundary of $\partial \Sigma^i$. We do this in a precise manner below.\\ 
First we describe the notion of smoothness of a connection $\omega$ on the bundle $E \times \Sigma \to B \times \Sigma$. Choose an open subset $U \subseteq B$ over which $E|_U \to U$ is trivial. Let $\{\phi^\mu : U \to E|_U\}$ for $\mu \in \{ 1,\ldots,n\}$ be a frame of smooth sections of this vector bundle. Let $\{\omega^\mu_\nu\}$ be the connection 1-forms corresponding to the connection $\omega$ with respect to this frame. We say that $\omega$ is a smooth connection if each $\omega^\mu_\nu$ is a  smooth 1-form. This definition depends neither on the choice of $U$ nor on the frame. For background on the notions of smoothness of forms on stratifolds, we refer the reader to the book \cite{kreck}, or to C.-O. Ewald's works \cite{coewald, coewaldthesis}.\\
Now, let $0 < \epsilon <1 $ be a small real number. First choose a diffeomorphism $\eta : \Delta^r \to \dr$ of smooth spaces where $\dr$ denotes the unit ball of dimension $r$. This diffeomorphism induces an isomorphism of bundles $E \times \dr \to B \times \dr$ and $E \times \Delta^r \to B \times \Delta^r$. Let $D^i_{disk}$ be the corresponding connection on $E \times \dr \to B \times \dr$. Now choose a smooth function $f : I \to I$ such that $f(t)=t\ \  \forall t \in [0,1-\epsilon)$, and $f(t)=1 \ \ \forall t \in [1-\frac{\epsilon}{2},1] $. The map $f$ induces a map $\tilde{f} : \dr \to \dr$ which sends the point $(r,\theta_1,\cdots,\theta_{r-1}) $ to $(f(r),\theta_1,\cdots,\theta_{r-1})$ in the polar coordinates. Let $D^{i,f}_{disk}$ be the connection on the bundle $E \times \dr \to B \times \dr$ given by $D^{i,f}_{disk}= (id \times \tilde{f})^{*}(D^i_{disk})$. Finally let $D^i_f = (id \times \eta)^*D^{i,f}_{disk}$ be the connection on $E \times \Delta^r \to B \times \Delta^r$. By construction, this connection is constant in a small collar around the boundary $E \times \partial \Delta^r$. The connections $D^i_f$ patch together to give a smooth connection $D_f$ on $E \times \Sigma \to B \times \Sigma$.\\
We now apply the Cheeger-Simons theory to the bundle $E \times \Sigma \to B \times \Sigma$ with connection $D_f$ to obtain a differential character $h^{f}_{B \times \Sigma} \in \hhat^{2p}(B \times \Sigma)$. Notice that the original Cheeger-Simons construction \cite{cs2} was done for the case when the base is a manifold. In our case however, $B \times \Sigma$ is a stratifold. This is not a problem however for us, because all that goes into the proof of Cheeger-Simons theorem is that the given connection arises from the pull back of a map from the bundle to an $m$-classifying Stiefel bundle ($E_N \to A_N$) with its canonical connection. While this may or may not be true when the base is a general stratifold  or for a general connection, for our purposes it suffices to show this for the connection $D_f$ over our bundle whose base is a cartesian product of manifold and geometric realization of a simplicial complex. We do this below.\\
Consider $\Delta^r$ as subset of $\mr^r$. Given a small collar $V$ around the $\partial \Delta^r$, and a classifying map $g_i : B \times \partial \Delta^r \to A_N$ for the connection $D^i_f|_{E \times \partial \Delta^r}$, we first extend this map to the collar (and use the same symbol $g_i$ for the extension) $g_i : B \times V \to A_N$. Then there   
 exists a smooth extension $\tilde{g}_i: B \times \Delta^r \to A_N$ of the  restriction $g_i|_{V'}$ of $g_i$ to a possibly smaller collar $V'\subseteq V$, such that $\tilde{g}_i^*(\gamma_N)=D^i_f$. (This follows by noticing that the Narasimhan-Ramanan proof proceeds by solving the existence of classifying map locally, and then using a partition of unity argument to get a global map.) Therefore we can assume that the maps $\tilde{g}_i : B \times \Sigma^i \to A_N $ agree with each other on the intersecting boundaries of the faces $\{ \Sigma^i \}$ and give rise to a smooth map $\tilde{g} : B \times \Sigma \to A_N$ under which the canonical connection $\gamma_N$ on the Stiefel bundle $E_N \to A_N$ pulls back to $D_f$. This fact enables us to apply Cheeger-Simons theory to the data $(E \times \Sigma \to B \times \Sigma, D_f)$ to get a differential character $h^f_{B \times \Sigma}$.\\ 
 Since $\Sigma$ is a cycle, its geometric realization is a compact stratifold without boundary. Therefore, by remark  \ref{sfold}, we can apply fiber integration to the bundle $B \times \Sigma \to B$, to obtain a differential character $\dcfib{\Sigma}(h^f_{B \times \Sigma})$ on B. (Notice that, as remarked in \cite{jni}, geometric realization of $\Sigma^i$ (or $\Delta^r$)is not boundary-less and hence fiber integration of differential characters can not be applied to the bundle $B \times \Sigma^i \to B$. However, since $\Sigma$ is a cycle, the geometric realization of $\Sigma$  has no boundary. This is what makes it possible to apply fiber integration to the bundle $B \times \Sigma \to B$). We have thus obtained a map : \begin{equation}\label{invdef}
\psi^f_{(P,u,r)} : \mathcal{Z}_r(\DE) \to \hhat^{2p-r}(B)                                                                                                                                                                                                                                                                                                                                                                                                                                                                                                                                                                                                                                                                                                                                                                                    \end{equation} given by 
                                                                                                                                                                                                                                                                                                                                                                                                                                                                                                                           \begin{equation}
\Sigma \mapsto \dcfib{\Sigma}(h^f_{B \times \Sigma})
                                                                                                                                                                                                                                                                                                                                                                                                                                                                                                                           \end{equation}
                                                                                                                                                                                                                                                                                                                                                                                                                                                                                                                           We often drop one or more of the subscripts $P,u$ and $r$ so as to avoid cluttered notation.\\ 
Let us now compute the curvature of the differential character $\psi^f(\Sigma)$. To do this, we note that by proposition 46, and equation (62) of \cite{bb}, we have a commutative diagram : \\
\begin{tikzcd}
0 \arrow[r] & {H^{k-1}(B \times \Sigma, \mathbb{R}/\mathbb{Z})} \arrow[dd, "\pi_{!\Sigma}"] \arrow[rr] &  & \hhat^k(B \times \Sigma) \arrow[rr] \arrow[dd, "\dcfib{\Sigma}"] &  & \Omega^k(B \times \Sigma) \arrow[r] \arrow[dd, "\underset{\Sigma}{\fint}"] & 0 \\
 &  &  &  &  &  &  \\
0 \arrow[r] & {H^{k-r-1}(B,\mrz)} \arrow[rr] &  & \hhat^{k-r}(B) \arrow[rr] &  & \Omega^{k-r}(B) \arrow[r] & 0
\end{tikzcd}
Thus we have 
\begin{equation}\label{curvcalc}
curv(\psi^f(\Sigma)=\underset{\Sigma}{\fint} curv(h^f_{B \times \Sigma})=\underset{\Sigma}{\fint} P(\Omega^f)
\end{equation} where $\Omega^f$ is the curvature of the connection $D_f$ i.e. $\Omega^f=dD_f + D_f \wedge D_f$. (Strictly speaking, we should write$(\Omega^f)^\mu_\nu=d(D_f)^\mu_\nu + (D_f)^\mu_\lambda \wedge (D_f)^\lambda_\nu$ where $\{(D_f)^\mu_\nu\}$ are connection 1-forms on $U \times \Sigma$ with respect to a frame $\{\phi^\mu\}$ over a trivializing subset $U \subseteq B$. However, in order to avoid clutter of notation, we use the same symbol $D_f$ to denote the connections, and the connection 1-form w.r.t. a frame $\{\phi^\mu\}$ and omit the indices $\mu$, and $\nu$.) \\
Now we use the fact that for each $t\in \Sigma$, the restriction of the connection $D_f$ to the bundle $ E \times \{ t \} \to B \times \{ t \}$ is flat, together with a standard argument in the literature (see for example, \cite{lp,guruprasad,jni}) to prove below that when $p>r+1$ or when $p<r$, the curvature of $\psi^f(\Sigma)$ is zero.\\
First note that forms on a product manifold can be decomposed as  $\Omega^{k}(M \times N) = \underset{0 \leq l \leq k}{\bigoplus} \Omega^{l,k-l}(M \times N)$.\\

On $ E \times \Sigma^i \to B \times \Sigma^i$, the connection $D_f$ restricts to $D^i_f = \sum c_j(t_0,\ldots,t_r)D^{j,i}$. Now if $\Omega^i_f$ is the curvature of $D^i_f$ (or equivalently the restriction of $\Omega^f$ to $\Sigma^i$), then we have $\Omega^i_f = d_{\SigmaiB}D_f^i+D^i_f \wedge D^i_f$.\\
We now use local coordinates $(x_\alpha)_{1\leq \alpha \leq m}$on on open set $U \subset B$. This gives us a chart for $U \times \Sigma^i$ with coordinates $(x_1,\ldots,x_m,t_0,\ldots,t_{r-1})$. In these coordinates, we have
\begin{equation}
 d_{\SigmaiB}D^i_f=\underset{j,k}{\sum} \frac{\partial c_j}{\partial t_k} dt_k \wedge D^{j,i} + \underset{\alpha}{\sum}  \frac{\partial}{\partial x_\alpha}D^i_f \wedge dx_\alpha
\end{equation} \\
Thus $\Omega^i_f=\underset{j,k}{\sum} \frac{\partial c_j}{\partial t_k} dt_k \wedge D^{j,i} + \underset{\alpha}{\sum}  \frac{\partial D^i_f}{\partial x_\alpha} \wedge dx_\alpha + D^i_f \wedge D^i_f = \underset{j,k}{\sum} \frac{\partial c_j}{\partial t_k} dt_k \wedge D^{j,i}$.\\ The last equality follows since the restriction of $D^i_f$ to any slice $ E \times \{ t\} \to B \times \{ t\}$ gives a flat connection, thereby making the sum of the last two terms vanish. Therefore $\Omega_f \in \Omega^{1,1}(B \times \Sigma,\glnr)$. Since $P$ is a homogeneous polynomial of degree $p$, we have $P(\Omega^f) \in \Omega^{p,p}(B \times \Sigma)$. Thus \begin{equation}\label{curvvanishes}
P(\Omega^f) =0\ \ for \ \ p>r                                                                                                                                                                                                                                                                                                                                                                                                                                                                                                                                                                                                                                                                                                                                                                                                                                                                                                                           \end{equation}
The above argument is standard, see for example \cite{lp,jni,guruprasad}.\\
In the $p < r$ case, even though $P(\Omega^f)$ need not vanish, the integral $\fintsigma P(\Omega^f)$ still vanishes. We have thus obtained 
\begin{theorem}
 When $p \neq r$, $curv(\psi^f(\Sigma))=0$, and hence $\psi^f$ takes values in $H^{2p-r-1}(B,\mrz)$.
\end{theorem}
When $p=r$, the above map gives a differential character in $\hhat^{r}(B)$ rather than an $\mrz$ cohomology class.
We now proceed to show that $\psi$ evaluated on boundaries vanishes, and hence it gives rise to a map $\tilde{\psi}^f:\mathbb{H}_r(\DE) \to H^{2p-r-1}(B,\mrz)$. To see this, let $\Sigma=\partial K$ where $K$ is an ($r+1$)-chain in the simplicial group $\DE$ of relatively flat connections. The geometric realization of $\Sigma$ is the boundary of the geometric realization of $K$, and hence $h^f_{B \times \Sigma}$ is the restriction of $h^f_{B \times K}$ along its boundary. Therefore, we apply theorem \ref{dcfibstokes} to the bundle $E \times K \to B \times K$ to obtain : \begin{align}
\psi^f(\Sigma) &= \dcfib{\Sigma}(h_{B \times \Sigma})\\
               &=\iota ((-1)^{2p-dim K}\underset{K}{\fint}curv (h_{B \times K}))\\
               &=(-1)^{2p-dim K} \iota (\underset{K}{\fint} P(\Omega_{B \times K}))
                                                                                                                                                                                                                                                                                                                                                                                                                                                                                                                                                               \end{align}
                                                                                                                                                                                                                                                                                                                                                                                                                                                                                                                                                               Now $P(\Omega_{B \times K}) \in \Omega^{p,p}(B \times K)$, and $dim K=r+1$. Hence if $p \neq r+1$,\\ $\psi_P(\Sigma)=0 \ \forall \Sigma \in \mathcal{B}_r(\DE)$. We have thus obtained
                                                                                                                                                                                                                                                                                                                                                                                                                                                                                                                                                               \begin{theorem}
When $p\neq r,r+1$, the map $\psi^f$ induces a map $\tilde{\psi}^f_{(p,r)} : \mathbb{H}_r(\DE) \to H^{2p-r-1}(B,\mrz)$.                                                                                                                                                                                                        
                                                                                                                                    \end{theorem}                                                                                                                                     We now proceed to show that the map $\psi^f$ does not actually depend on the choice of the map $f$ of the unit interval or $\epsilon$. Let $f_0,f_1$ be any two such maps for $\epsilon_0$ and $\epsilon_1$ respectively. Let $F: I \times I \to I$ be given by $F(t,s)=(1-s)f_0(t)+sf_1(t)\equiv f_s(t)$. One can then define a connection on the bundle $E \times \Sigma \times I \to B \times \Sigma \times I$ which is trivial in the directions of $\Sigma \times I$, and whose restriction to the slice $E \times \Sigma \times \{s\} \to B \times \Sigma \times \{s\}$ is $D_{f_s}$. By theorem (\ref{dcfibstokes}) we have
                                                                                                                                    \begin{equation}
                                                                                                                                     \psi^{f_1}(\Sigma) - \psi^{f_0}(\Sigma) =(-1)^{r+1}\iota( \underset{\Sigma \times I}{\fint} curv(h^F_{B \times \Sigma \times I}) )
                                                                                                                                    \end{equation}
                                                                                                                                    Since $curv(h^F_{B \times \Sigma}) \in \Omega^{p,p}(B \times (\Sigma \times I))$, and since $p\neq r+1$, the R.H.S. vanishes, thereby proving the result : 
                                                                                                                                    \begin{theorem}\label{mapchoice}
                                                                                                                                     When $p\neq r+1$, $\psi_{p,r}^{f}(\Sigma)$ does not depend upon the choice of the function $f$.
                                                                                                                                    \end{theorem}
                                                                                                                                    We therefore omit the superscript $f$, and simply write $\psi$ or $\tilde{\psi}$.

\section{Comparison with other constructions in the literature}\label{comparison}
In this section, we first derive an explicit computable formula for $\psi(\Sigma)$. We find that this matches with the formula given by \cite{jni} thereby showing that the maps $\rho'$ in \cite{jni} are equal to the maps $\tilde{\psi}$ here for the $p > r+1$ case, and are zero for the $p<r$ case. We then show that as a consequence of theorem (\ref{dcfibstokes}), our invariants match with the ones constructed in \cite{bl} using entirely different methods. We further show that the results here are compatible with the ones obtained in \cite{lp}.\\ 
Fix any flat connection $D_0$ on $E \to B$, and let $\bar{D_0}$ be the connection on $E \times \Sigma \to B \times \Sigma$ obtained by pulling back the connection $D_0$ on $E \to B$ under the pull back diagram \begin{center} 
\begin{tikzcd}
E \times \Sigma \arrow[dd] \arrow[rr] &  & E \arrow[dd] \\
 &  &  \\
B \times \Sigma \arrow[rr] &  & B
\end{tikzcd}
\end{center}

Now since the space of connections is convex, there exists a connection $\tilde{D}$ on $E \times \Sigma \times I \to B \times \Sigma \times I$ such that its restriction to $E \times \Sigma \times \{0\}\to B \times \Sigma \times \{0\}$ is $\bar{D}$, and restriction to the slice $E \times \Sigma \times \{1\} \to B \times \Sigma \times \{1\}$ is $D$. Note that the restriction of $\tilde{D}$ to the slice $ E \times \{p\} \times \{t\}\to B \times \{p\} \times \{t\}$ for $0<t<1, p\in \Sigma$ need not be flat. \\
By theorem \ref{dcfibstokes}, we have 
\begin{equation}
 h^D_{B \times \Sigma \times \{1\}} - h^D_{B \times \Sigma \times \{0\}} = (-1) \iota (\underset{I}{\fint} curv (h^{\tilde{D}}_{B \times \Sigma \times I})) \in \hhat^{2p-r}(B \times \Sigma)
\end{equation}
Applying further integration over the fiber $\Sigma$ in the bundle $B \times \Sigma \to B$, we get 
\begin{equation}\label{explicitformula}
 \psi(\Sigma) = (-1)^{r+1} \iota (\underset{\Sigma \times I}{\fint} P(\Omega^{\tilde{D}}))
\end{equation}
This is precisely\footnote{In \cite{jni}, the results are stated for the case $p > r$. However, to the best of our understanding, this is an error and the results there hold for $p > r+1$. In the case $p=r+1$, one gets maps $Z_r(\DE) \to H^{2p-r-1}(B,\mrz)$, however they need not necessarily descend to maps on homology. Also compare remark 7 in \cite{lp}.} the construction given in \cite{jni} (since the R.H.S. in \ref{explicitformula} does not depend upon $\epsilon$, we can take $\epsilon \to 0+$). Therefore, we obtain : 
\begin{theorem}\label{expthm}
 The maps $\tilde{\psi} : \mathbb{H}_r(\DE) \to H^{2p-r-1}(B,\mrz)$ are equal to the maps $\rho'$ constructed in \cite{jni} for $p\geq r+2$.
\end{theorem}
Invariance under the action of the structure group $GL(n,\mr)$ follows from the invariance of $P$ and equation \ref{explicitformula}.
Notice that the equation \ref{explicitformula} shows that $\psi_{p,u}$ does not depend on $u$. Further it shows that that a different choice of the path of simplices connecting the trivial simplex to $D$ changes the integral in the R.H.S. of equation \ref{explicitformula} at most by a closed form with integral periods. In \cite{jni} R.H.S. is taken as the definition of the invariant associated to the family of flat connections.\\
Note that $\tilde{\psi}([\Sigma])$ considered as an element of $\hhat^{2p-r}(B)$ lies in the image of both the inclusion maps $H^{2p-r-1}(B,\mrz) \to \hhat^{2p-r}(B)$, and $\frac{\Omega^{2p-r-1}}{\Omega^{2p-r-1}_0} \xrightarrow{\iota} \hhat^{2p-r}(B)$. Thus not only does the curvature of this differential character vanish, but also its characteristic class. The vanishing of the characteristic class can also be understood in the following way.
\begin{theorem}\label{chclassvanishes}
 The characteristic class of the differential character $\tilde{\psi}([\Sigma])$ vanishes. 
\end{theorem}
\begin{proof}
By compatibility of fiber integration of differential characters, and singular cohomology we have that $ch(\dcfib{\Sigma}(h_{B \times \Sigma}))=\pi_{!\Sigma} ch(h_{B \times \Sigma})$. Now we have an explicit description of the fiber integration map in singular cohomology (see the discussion in remark 4.5 of \cite{bb}). If $\mu \in H^{2p}(B \times \Sigma,\mz)$, then $\pi_{!\Sigma}(\mu) \in H^{2p-r}(B,\mz)$ is given by $\pi_{!\Sigma}(\mu)(c) = \mu (EZ(c \otimes \Sigma))$ where $\Sigma$ denotes the fundamental class of the top homology $H_r(\Sigma,\mz) \simeq \mz$. Now observe that if $u\in H^{2p}(G(k,\infty))$ is an element of cohomology of the base of the classifying bundle, the characteristic class $\mu$ of the bundle $E \times \Sigma \to B \times \Sigma$ is given by $\mu(c')=u(f_*(c'))$ where $f: B \times \Sigma \to G(k,\infty)$ is any classifying map. In our case, if $g : B \times \{*\} \to G(k,\infty)$ is a classifying map for the bundle $E \times \{*\} \to B \times \{*\} $, we can choose $f= g \circ (id_B \times \{*\})$. Thus $f_*(EZ(c \otimes \Sigma))=EZ( g_* (c) \otimes 0) = 0$. Hence we have $\pi_{!\Sigma}(ch(h_{B \times \Sigma}))=0$.
\end{proof}
Since $ch \circ i_1$ is (upto a sign) the connecting homomorphism in the long exact sequence in cohomology corresponding to the short exact sequence $0 \to \mz \to \mr \to \mrz \to 0$, we have the :
\begin{corollary}\label{liftexists}
 For any $[\Sigma] \in H_r(\DE)$, $\tilde{\psi}([\Sigma])$ lies in the image of the map $H^{2p-r-1}(B,\mr) \to H^{2p-r-1}(B,\mrz)$.
\end{corollary}

Our initial hope was that this method of fiber integration would yield new cohomological invariants for the $p < r$ case, since all that is required for our construction of $\tilde{\psi}$ is that $p \neq r,r+1$. However a closer look reveals that for $p<r$, $\tilde{\psi}([\Sigma])=0$. Therefore this yields no nontrivial invariants in this case. 
\begin{theorem}
 For $p<r$, $\psi({\Sigma})=0 \ \ \forall \  \Sigma \in \mathcal{Z}_r(\DE) $
\end{theorem}
\begin{proof}
$ \Omega^{\tilde{D}}=\Omega^{\tilde{D}\ 2,0} + \Omega^{\tilde{D}\ 1,1} + \Omega^{\tilde{D}\ 0,2} \in \Omega^2(B \times (\Sigma \times I))$. However, as has been argued above, $\Omega^{\tilde{D}\ 0,2}=0$ since $\tilde{D}$ is trivial in the directions of $\Sigma \times I$. Thus 
 \begin{equation}
  P(\Omega^{\tilde{D}}) \in \Omega^{p,p}(B \times (\Sigma \times I)) \oplus \Omega^{p+1,p-1}(B \times (\Sigma \times I)) \oplus \cdots \oplus \Omega^{2p,0}(B \times (\Sigma \times I))
 \end{equation}
Since $dim (\Sigma \times I) = r+1$, the R.H.S. in equation \ref{explicitformula} vanishes. 
\end{proof}
Though this result regarding vanishing of invariant for the $p < r$ case has not been explicitly stated in \cite{jni}, it is a consequence of Jaya Iyer's formula i.e. equation \ref{explicitformula} in this paper.\\
We now proceed to compare our results with others in the literature obtained using different methods. We have already seen that the characters constructed using fiber integration coincide with the ones constructed in \cite{jni}. We see below that the invariants are equal to the ones constructed in \cite{bl,lp} upto a sign. \\
In \cite{bl} the authors consider the set of smooth maps $Maps(S,\mathcal{F})$ where $S$ is a null cobordant manifold and $\mathcal{F}$ denotes the space of all flat connections. The space of all connections on $E \to B$ is denoted by $\mcA$ and is given the structure of a Fr\'echet manifold. Employing the framework of Atiyah bundle, and the bundle of connecions, they construct certain ($r+1$)-differential forms $\beta^p_{r+1} \in \Omega^{r+1}(\mcA,\Omega^{2p-r-1}(B))$ on $\mcA$ taking values in $\Omega^*(B)$. They then define maps $\Lambda^p_{r+1} : Maps(S,\mcF) \to H^{2p-r-1}(B,\mr)$ given by $f \mapsto \Lambda^p_{r+1}(f)=\underset{T}{\int} \bar{f}^* \beta^p_{r+1}$ where $\bar{f} : T \to \mcF$ is any extension of the given map $f: S \to \mcF$ to a manifold $T$ whose boundary is $S$. They show that these maps are well defined i.e the R.H.S. does not depend (considered as an element of cohomology group $H^{2p-r-1}(B,\mr)$) upon the choice of $T$ or $\bar{f}$. As a consequence of certain identities in their paper, they further show that if $f_0,f_1 \in Maps(S,\mcF)$ are homologous, then $\Lambda^p_{r+1}(f_0)=\Lambda^p_{r+1}(f_1)$. Also they prove that 
\begin{equation}\label{blidentity}
\Lambda^p_{r+1}(f) \equiv \underset{T}{\int} \bar{f}^* \beta^p_{r+1} = \underset{T}{\fint} P(\Omega)
\end{equation}
 where $\Omega$ is the curvature of the connection induced on the bundle $T \times E \to T \times B$.\\
From the standpoint of the theory of differential characters and fiber integration used in this paper, we find that their results (modulo $\mz$) follow from an application of Theorem \ref{dcfibstokes}. We could let 
\begin{equation}
 \Lambda^p_{r+1} : Maps(S,\mcF) \to H^{2p-r-1}(B,\mrz), \ \ p\neq r,r+1
\end{equation} given by 
\begin{equation}
 f \mapsto (-1)^{2p-r-1} \dcfib{S}(h_{S \times B})
\end{equation}
where $h_{S \times B}$ is the differential character corresponding to the connection on the bundle $E \times S \to B \times S$. (We have chosen the factor $(-1)^{2p-r-1}$ so as to take care of the signs.)
Equation \ref{blidentity} is then a corollary of theorem \ref{dcfibstokes}. The fact that the form $\Lambda^p_{r+1}(f)$ is closed is automatic since this form is the curvature of the differential character $\dcfib{S}(h_{S \times B})$ by virtue of equation \ref{curvcalc}. Note that Lopez and Biswas obtain maps into $H^{2p-r-1}(B,\mr)$ instead of $H^{2p-r-1}(B,\mrz)$. From our standpoint, this is because the differential character has trivial characteristic class (see corollary \ref{liftexists}). Invariance under the homology relation (proposition 3.8 of \cite{bl}) is a consequence of theorem \ref{dcfibstokes} by an argument similar to the one used in proving theorem \ref{mapchoice}. Their invariants for $p<r$ vanish because the forms $\beta^p_{r}$ themselves vanish. From our point of view, even though $P(\Omega_{B \times T})$ does not vanish, $\underset{T}{\fint} P(\Omega_{B \times T})$ vanishes. Also note that their construction does not require a choice of $u \in H^{2p}(BGL(n,\mr),\mz)$ in the first place, while ours turns out to be independent of this choice (see the discussion after theorem \ref{expthm}). Also compare this discussion with remark 3.4 in \cite{bl} where the authors show that $\Lambda^p_{r+1}$ can be expressed as fiber integrals of certain transgression forms which can be used to prove some of their results.\\
As a special case, they consider the case where $S=S^r$, and obtain the maps $\Lambda^p_{r+1} : \pi_{r}(\mathcal{F}) \to H^{2p-r-1}(B,\mr)$ from homotopy groups of $\mathcal{F}$ to $H^{2p-r-1}(B,\mr)$. In contrast to homotopy groups, general homology classes can not be represented by maps from smooth manifolds, which is why we have used stratifolds in this paper. \\
In \cite{lp}, the authors use the theory of differential characters and equivariant characteristic classes to obtain maps $H_{2p-r-1}(B,\mz) \times H_{r}(\mcF/\mathcal{G},\mz) \to \mrz$ for $p>r+1$ where $\mathcal{G} \subset Gau(P)$ is a subgroup of the group of global gauge transformations which acts freely on $\mathcal{A}$. Their approach is to use the canonical connection $\mathbb{A}$ on the bundle $E \times \mathcal{A} \to B \times \mathcal{A}$ which is trivial in the directions of $\mathcal{A}$, and whose restriction to $E \times \{ \theta \} \to B \times \{ \theta \}$ is $\theta$. They then choose a connection $\mathcal{U}$ on the principal bundle $\mathcal{A} \to \mathcal{A}/\mathcal{G}$. They show that this data determines a connection $\underline{\mathcal{U}}$ on the bundle $(E \times \mathcal{A})/\mathcal{G} \to B \times \mathcal{A}/\mathcal{G}$. Cheeger-Simons theory on this bundle yields a character $\chi_{\underline{\mathcal{U}}} \in \hhat^{2p}(B \times \mathcal{A}/\mathcal{G})$. Their maps are the composition $Z_{2p-r-1}(B,\mz) \times Z_{r}(\mathcal{F}/\mathcal{G},\mz) \to Z_{2p-1}(B \times \mathcal{F}/\mathcal{G}, \mz ) \overset{\chi_{\underline{\mathcal{U}}}}{\longrightarrow} \mrz$. The first map here is the standard multiplication map. They then show that these maps vanish when either factor is a boundary, and hence descend to homology groups $ \chi_{p,r}: H_{2p-r-1}(B,\mz) \times H_{r}(\mathcal{F}/\mathcal{G}) \to \mrz$. They further show that these maps do not depend on the choice of the connection $\underline{\mathcal{U}}$. \\
Our method is a variant of their approach, and has the disadvantage that it does not prove the invariance of the invariants under the action of the Gauge group (or more precisely, a subgroup $\mathcal{G}$ of $Gau(E)$ which acts freely on $\mathcal{A}$). Thus we have to restrict our attention to obtaining maps $H_r(\mathcal{F}) \to H^{2p-r-1}(B,\mrz)$ rather than $H_r(\mathcal{F}/\mathcal{G}) \to H^{2p-r-1}(B,\mrz)$. In this sense our results are a special case of the results of \cite{lp} for the choice $\mathcal{G}=\{e\}$. (Notice that nothing in our approach actually uses the fact that the connection on the bundle corresponding to a point on the geometric realization of the cycle is a linear combination of the connection corresponding to the vertices. All that is required is that the connection corresponding to any point in the parameter space is flat.) However our approach does not assume that the Cheeger-Simons construction is applicable in the case when the base is an infinite dimensional Fr\'echet manifold. The Cheeger-Simons construction of a differential character given a bundle with a connection uses the fact (due to Narasimhan-Ramanan \cite{narsimhanramanan}) that any smooth connection on the bundle $E \to B$ (with the base $B$ a finite dimensional manifold) can be obtained from a universal connection on a Stiefel bundle by pullback under a smooth classifying map. However to the best of knowledge of the present author, an analogous result when the base is an infinite dimensional Fr\'echet manifold has not yet been proved in the literature (also see the discussion in section 2.1.1, and footnote 3, p.9 in \cite{becker}).\\
Notice also that our result that the character $\tilde{\psi}(\Sigma)$ does not depend on the choice of $u$, proves that the maps $\chi_{P,u} : H_r(\mathcal{F},\mz) \times H_{2p-r-1}(B,\mz) \to \mrz$ constructed in \cite{lp} do not depend upon the choice of $u$. We do not expect this to be true more generally in the case when $\mathcal{G} \neq \{e\}$.

We have already seen that the maps constructed in \cite{jni} are compatible with the maps in \cite{bl}. In proposition 10 of \cite{lp}, it is shown that the maps constructed there are compatible with the ones in \cite{bl}. It is also possible to understand the relation of our construction with the one in \cite{lp}. For this purpose we need the following lemma which follows directly from the definition of fiber integration. 
\begin{lemma}
 Let $h \in \hhat^k(E)$, and let $f : \Sigma \to E$ represent a $(k-1)$-cycle $z$. Then 
 \begin{equation}
  h(z) = \dcfib{\Sigma}(f^*h)
 \end{equation}
where the fiber integration is along the fiber $\Sigma$ of the fiber bundle $\Sigma \to *$.
\end{lemma}
Let $\chi_{\mathbb{A}} \in \hhat^{2p}(B \times \mathcal{A})$ be the character corresponding to the connection $\mathbb{A}$ on $E \times \mathcal{A} \to B \times \mathcal{A}$. Let $f : \Sigma \to \mathcal{F} \to \mathcal{A}$ be a cycle representing a homology class in $H_r(\mathcal{F})$. Then the connection that we used in our construction is the pull back of the connection $\mathbb{A}$ in the diagram \begin{center}
                                                                                                                                                                                                                                                            
\begin{tikzcd}
E \times \Sigma \arrow[r] \arrow[d] & E \times \mathcal{A} \arrow[d] \\
B \times \Sigma \arrow[r] & B \times \mathcal{A}.
\end{tikzcd} 
\end{center}
Hence $h_{B \times \Sigma} \in \hhat^{2p}(B \times \Sigma)$ is given by $h_{B \times \Sigma}= (id_B \times f)^* \chi_{\mathbb{A}}$. Let $g : K \to B$ is a $(2p-r-1)$-cycle in $B$. Then consider $g \times f : K \times \Sigma \to B \times \mathcal{F}$
We then have  
\begin{align*}
 \psi(\Sigma)(K) &= \dcfib{\Sigma}(h_{B \times \Sigma})(K)\\
                &= \dcfib{K} (g^*(\dcfib{\Sigma}(h_{B \times \Sigma})))\\
                &=\dcfib{\Sigma \times K} ((g \times f)^*\chi_{\mathbb{A}})\\
                &= \chi_{\mathbb{A}}(\Sigma \times K).
\end{align*}
\section*{Acknowledgments}
I am thankful to my supervisor Dr. Rishikesh Vaidya for discussions and support. I am financially supported by the Council of Scientific \& Industrial Research - Human Resource Development Group (CSIR-HRDG) under the CSIR-SRF(NET) scheme. I am grateful to CSIR-HRDG for the same.
\bibliography{refs}
Ishan Mata \textsc{Department of Physics, Birla Institute of Technology and Science - Pilani, Pilani campus, Pilani, Rajasthan, India. PIN : 333031}\\
\textit{Email} : ishanmata@gmail.com
\end{document}